\documentclass[11pt]{amsproc}

\usepackage{amssymb, amsfonts, epsfig}
\newtheorem{Thm}{Theorem}[section]
\newtheorem{corollary}[Thm]{Corollary}
\newtheorem{lemma}[Thm]{Lemma}
\newtheorem{proposition}[Thm]{Proposition}
\newtheorem{definition}[Thm]{Definition}
\newtheorem{remark}[Thm]{Remark}

\newtheorem{example}[Thm]{Example}

\newtheorem{theorem}[Thm]{Theorem}
\newtheorem{problem}[Thm]{Problem}
\newcommand{\norm}[1]{\|#1\|}
\DeclareMathOperator{\lspan}{span}
\DeclareMathOperator{\codimm}{codim}
\DeclareMathOperator{\dimm}{dim}
\DeclareMathOperator{\spark}{spark}

\begin{document}

\title{Phase Retrieval}
\author[J. Cahill, P.G. Casazza, J. Jasper, L. Woodland]{Jameson Cahill, Peter G. Casazza, John Jasper, and Lindsey M. Woodland}
\address{Department of Mathematics, Duke University, Durham, NC, USA}
\address{Department of Mathematics, University of Missouri, Columbia, MO, USA}
\address{Department of Mathematical Sciences, 4199 French Hall West, University of Cincinnati, 2815 Commons Way, Cincinnati, OH, USA}
\address{435 North Oak Park Ave., Unit 2, Oak Park, IL, USA}


\email{J.C. Email:jameson.cahill@gmail.com, P.G.C Email:casazzap@missouri.edu (Send correspondence to P.G.C.), J.J. Email:john.jasper@uc.edu, L.M.W. Email:lmwvh4@gmail.com}

\begin{abstract}
We answer a number of open problems concerning phase retrieval and phase retrieval by projections. In particular, one main theorem classifies phase retrieval by projections via collections of sequences of vectors allowing norm retrieval. Another key result computes the minimal number of vectors needed to add to a frame in order for it to possess the complement property and hence allow phase retrieval. In furthering this idea, in a third main theorem we show that when a collection of subspaces is one subspace short from allowing phase retrieval, then any partition of orthonormal bases from these subspaces into two sets which fail to span, then each spans a hyperplane. We offer many more results in this area as well as provide a large number of examples showing the limitations of the theory.
\end{abstract}


\maketitle

\section{Introduction}

In the setting of frame theory, the concept of phaseless reconstruction was introduced in 2006 by Balan, Casazza, and Edidin \cite{Balan06}. At that time, they classified phaseless reconstruction in the real case by proving that a {\it generic} family of (2N-1)-vectors in $\mathcal{R}^N$ does phaseless reconstruction and no set of (2N-2)-vectors can do this.  In the complex case, they showed that a {\it generic} set of (4N-2)-vectors does phaseless reconstruction. Since then there has been a multitude of mathematical research devoted to this area. One area in particular is the study of phaseless reconstruction by projections onto subspaces. An in depth study of phase retrieval by projections was done by Cahill, Casazza, Peterson, and Woodland in \cite{CCPW} where they showed that phase retrieval by projections can be done in the real case with $(2N-1)$-projections of arbitrary non-trivial rank and in the complex case with $(4N-2)$-projections. We continue the study of phase retrieval by providing a variety of new results in both the one-dimensional and higher dimensional cases as well as giving a large number of examples showing the limitations of the theory. 

The complex case has proven to be a very difficult problem and as such has yet to be fully classified. There has been incremental progress towards this end and in particular Heinossaari, Mazzarella and Wolf \cite{HMW} showed that M-vectors doing phaseless reconstruction in $\mathcal{C}^N$ requires $M\ge 4N-4-2\alpha$, where $\alpha$ is the number of $1's$ in the binary expansion of $(N-1)$.  Bodmann \cite{BN} later showed that phaseless reconstruction in $\mathcal{C}^N$ can be done with $(4N-4)$-vectors.  Following this, Conca, Edidin, Hering, and Vinzant \cite{CEHV} proved that a {\it generic} frame with $(4N-4)$-vectors does phaseless reconstruction in $\mathcal{C}^N$.  They also showed that if $N=2^k+1$ then no set of $M$-vectors with $M< 4N-4$ can do phaseless reconstruction. Bandeira, Cahill, Mixon, and Nelson \cite{BCMN} conjectured that for all $N$, no fewer than $(4N-4)$-vectors can do phaseless reconstruction.  Recently, Vinzant \cite{V} showed that this conjecture does not hold by giving 11 vectors in $\mathcal{C}^4$ which do phase retrieval.  And recently, Xu \cite{X} showed that phase retrieval can be done in $\mathcal{R}^4$ with six 2-dimensional subspaces.

\section{Frame Theory}

This section first provides a basic review of necessary terms and theorems from finite frame theory and then goes on to develop numerous new results regarding frames, Riesz bases and complements. For a more in depth study of finite frame theory the interested reader is referred to {\it Finite Frames: Theory and Applications} \cite{petesbook}. To set notation, given $N\in \mathcal{N}$, $\mathcal{H}_{N}$ represents a $($real or complex$)$ Hilbert space of $($finite$)$ dimension $N$.

\begin{definition}
A family of vectors $\{\varphi_i\}_{i=1}^M$ in an $N$-dimensional Hilbert space $\mathcal{H}_N$ is a {\bf frame} if there are constants $0 < A \leq B < \infty$ so that for all $x \in \mathcal{H}_M$,
\begin{equation*} 
A \| x \|^2 \leq \sum_{i=1}^M |\langle x,\varphi_i\rangle|^2 \leq B \|x\|^2,
\end{equation*}
where $A$ and $B$ are {\bf lower and upper frame bounds}, respectively.  The largest $A$ and smallest $B$ satisfying these inequalities are called the {\bf optimal frame bounds}.

\begin{enumerate}

\item If $A = B$ is possible, then $\{\varphi_i\}_{i=1}^M$ is an {\bf A-tight frame}. If $A = B = 1$ is possible, then $\{\varphi_i\}_{i=1}^M$ is a {\bf Parseval frame}.

\item If $\|\varphi_i\| = 1$ for all $i \in [M]$ then $\{\varphi_i\}_{i=1}^M$ is an {\bf unit norm frame}. 

\item $\{\langle x,\varphi_i\rangle \}_{i = 1}^M$ are called the {\bf frame coefficients} of the vector $x\in \mathcal{H}_N$ with respect to frame $\{\varphi_i\}_{i=1}^M$.

\item The {\bf analysis operator} of the frame is 
$T:\mathcal{H}_N\rightarrow \ell_2^M$ given by
\[ T(x) = \{\langle x,\phi_i\rangle\}_{i=1}^M.\]

\item The {\bf frame operator S} of the frame, for any $x \in \mathcal{H}_N$, is given by
\[ Sx = \sum_{i=1}^M\langle x,\varphi_i\rangle \varphi_i.\]
\end{enumerate}
\end{definition}

The frame operator is a particularly important object in the study of frame theory. 

\begin{theorem}
Let $\{\varphi_{i}\}_{i=1}^M$ be a frame for $\mathcal{H}_N$ with frame bounds $A$ and $B$ and frame operator $S$. $S$ is positive, invertible, and self adjoint. Moreover, the optimal frame bounds of $\{\varphi_i\}_{i=1}^M$ are given by $A=\lambda_{min}(S)$ and $B=\lambda_{max}(S)$, the minimum and maximum eigenvalues of $S$, respectively. 
\end{theorem}

In particular, the frame operator of a Parseval frame is the identity operator. This fact makes Parseval frames very helpful in applications because they possess the property of perfect reconstruction. That is, $\{\varphi_i\}_{i=1}^M$ is a Parseval frame for $\mathcal{H}_N$ if and only if for any $x\in \mathcal{H}_N$ we have 
\[x=\sum_{i=1}^M\langle x,\varphi_i\rangle \varphi_i.\]

There is a direct method for constructing Parseval frames.  For $M \geq N$, given an $M\times M$ unitary matrix, select any $N$ rows from this matrix, then the column vectors from these rows form a Parseval frame for $\mathcal{H}_N$. Moreover, the leftover set of $M-N$ rows also have the property that its $M$ columns form a Parseval frame for $\mathcal{H}_{M-N}$. The following well known theorem, known as {\bf Naimark's Theorem}, utilizes this type of construction and shows that this is the only way to obtain Parseval frames.

\begin{theorem}[Naimark's Theorem; page 36 of \cite{petesbook}]\label{naimark}
Let $\Phi=\{\varphi_i \}_{i=1}^M$ be a frame for $\mathcal{H}_N$ with analysis operator $T$, let $\{e_i\}_{i=1}^M$ be the standard basis of $\ell_2\left(M\right)$, and let $P:\ell_2\left(M\right) \rightarrow \ell_2\left(M\right)$ be the orthogonal projection onto $\mbox{range}\left(T\right)$. Then the following conditions are equivalent:
\begin{enumerate}
\item $\{\varphi_i\}_{i=1}^M$ is a Parseval frame for $\mathcal{H}_N$.
\item For all $i=1,\dots,M$, we have $Pe_i=T\varphi_i$.
\item There exist $\psi_1,\dots, \psi_M \in \mathcal{H}_{M-N}$ such that $\{ \varphi_i \oplus \psi_i \}_{i=1}^M$ is an orthonormal basis of $\mathcal{H}_M$.
\end{enumerate}

Moreover, if (3) holds, then $\{\psi_i\}_{i=1}^M$ is a Parseval frame for $\mathcal{H}_{M-N}$. If $\{\psi'_i\}_{i=1}^M$ is another Parseval frame as in (3), then there exists a unique unitary operator $L$ on $\mathcal{H}_{M-N}$ such that $L\psi_i=\psi_i',$ for all $i=1,\dots,M$.
\end{theorem}

Explicitly, we call $\{\psi_i\}_{i=1}^M$ the {\bf Naimark Complement} of $\Phi$. If $\Phi = \{\varphi_i\}_{i=1}^M$
is a Parseval frame, then the analysis operator $T$ of the
frame is an isometry.  So we can associate 
$\varphi_i$ with $T\varphi_i = Pe_i$, and with a slight
 abuse of notation
we have:

\begin{theorem}[Naimark's Theorem]
$\Phi = \{\varphi_i\}_{i=1}^M$ is a Parseval frame for
$\mathcal{H}_N$ if and only if there is an $M$-dimensional Hilbert space $\mathcal{K}_M$ with an orthonormal
basis $\{e_i\}_{i=1}^M$ such that the orthogonal projection
$P:\mathcal{K}_M \rightarrow \mathcal{H}_N$ satisfies
$Pe_i= \varphi_i$ for all $i=1,\dots,M$.  Moreover, the Naimark complement of
$\Phi$ is $\{(I-P)e_i\}_{i=1}^M.$
\end{theorem}
 
Note that Naimark complements are only defined for Parseval frames.  Furthermore, Naimark complements  are only defined up to unitary equivalence. That is, if $\{\varphi_i\}_{i=1}^M \subseteq\mathcal{H}_N$ and $\{\psi_i\}_{i=1}^M\subseteq\mathcal{H}_{M-N}$ are Naimark complements, and $U$ and $V$ are unitary operators, then $\{U\varphi_i\}_{i=1}^M$ and $\{V\psi_i\}_{i=1}^M$ are also Naimark complements.

To clarify terminology, as mentioned in Naimark's Theorem and as will be throughout this paper, an {\it orthogonal projection} or simply a {\it projection} is a self-adjoint projection.

Frames in the finite dimensional setting are just spanning sets. At times, it is useful to look at subsets of a frame which are also spanning sets.

\begin{definition}
A frame $\{\varphi_i\}_{i=1}^M$ in $\mathcal{H}_N$ satisfies the {\bf complement property} if for all subsets $S \subset \{1,\dots, M\}$, either $\lspan \{\varphi_i\}_{i\in S}=\mathcal{H}_N$ or $\lspan \{\varphi_i\}_{i\in S^c}=\mathcal{H}_N$.
\end{definition}

\begin{definition}
Given a family of vectors $\Phi=\{\varphi_i\}_{i=1}^M$ in $\mathcal{H}_N$, the {\bf spark} of $\Phi$ is defined as the cardinality of the smallest linearly dependent subset of $\Phi$.  When $\spark(\Phi)=N+1$, every subset of size $N$ is linearly independent and $\Phi$ is said to be {\bf full spark}.
\end{definition}

\begin{remark}
Let $\Phi:=\{\varphi_i\}_{i=1}^M$ be a frame in $\mathcal{H}_N$ such that $M\geq 2N-1$.
\begin{enumerate}
\item If $\Phi$ is full spark then $\Phi$ has the complement property.
\item If $\Phi$ has the complement property and $M=2N-1$ then $\Phi$ is full spark. 
\end{enumerate}
\end{remark}

The notion of spark is the measure of how resilient a frame is against erasures, so {\it full spark} is a desired property of a frame. In general, it is very difficult to check the spark of a frame. In attempts to further characterize full spark frames, in \cite{CCJW} the authors classified full spark Parseval frames through the use of the frame's Naimark complement.

\begin{proposition}\cite{CCJW}\label{cor1314}
 A Parseval frame is full spark if and only if its Naimark
 complement is full spark.
 \end{proposition}

Often times in applications, a frame is linearly dependent and hence the decomposition of a signal with respect to a frame is not unique. However, it may be necessary to have a unique decomposition without restricting the frame to such properties as orthogonality. A \emph{Riesz basis} provides this uniqueness and does not have as strong of a condition as orthogonality.

\begin{definition}
A spanning family of vectors $\{\varphi_i\}_{i=1}^N$ in a Hilbert space $\mathcal{H}_N$ is called a {\bf Riesz basis} with \emph{lower} (respectively, \emph{upper}) \emph{Riesz bounds A} (respectively, \emph{B}), if, for all scalars $\{a_i\}_{i=1}^N$, we have
\[A\sum_{i=1}^N|a_i|^2 \leq \|\sum_{i=1}^Na_i\varphi_i\|^2 \leq B\sum_{i=1}^N|a_i|^2.\]
\end{definition}

Recall a result from \cite{BCPS}:
\begin{theorem}
Let $P$ be a projection on $\mathcal{H}_M$ with orthonormal basis
 $\{e_i\}_{i=1}^M$ and let $I\subset \{1,2,\ldots,M\}$.
The following are equivalent:
\begin{enumerate}
\item $\{Pe_i\}_{i\in I}$ is linearly independent.
\item $\{\left(I-P\right) e_i\}_{i\in I^c}$ spans $\left(I-P\right)\mathcal{H}$.
\end{enumerate}
\end{theorem}

As a consequence we have,

\begin{corollary}
Let $P$ be a projection of rank $N$ on $\mathcal{H}_M$ with orthonormal basis $\{e_i\}_{i=1}^M$. 
The following are equivalent:

(1)  $\{Pe_i\}_{i=1}^M$ has the complement property.

(2)  Whenever we partition $\{(I-P)e_i\}_{i=1}^M$ into two sets, one of them is linearly independent.
\end{corollary}

The previous results analyze when a collection of projections and their orthogonal complements are linearly independent and when they possess the complement property. A natural next step is to see when a projection of
an orthonormal basis is full spark.

\begin{proposition}
Let $P$ be a projection of rank $N$ on $\mathcal{H}_{M}$ with orthonormal basis $\{e_i\}_{i=1}^M$. 
The following are equivalent:

(1)  $\{Pe_i\}_{i=1}^M$ is full spark.

(2)  For every $I\subset \{1,2,\ldots M\}$ with $|I|=M-N$
the vectors $\{(I-P)e_{i}\}_{i\in I}$ spans $(I-P)(\mathcal{H})$.
\end{proposition}

\begin{proof}
This follows immediately from Proposition \ref{cor1314}.\end{proof}

Instead of only considering an orthonormal basis for our projections, the next proposition slightly generalizes this idea by using Riesz bases.

\begin{proposition} 
Let $\{\varphi_i\}_{i=1}^N$ be a Riesz basis with dual basis $\{\varphi_i^*\}_{i=1}^N$ for $\mathcal{H}_N$ and let $P$ be an orthogonal projection on $\mathcal{H}_N$ of rank $M$. Let $I\subset \{1,\dots, N\}$. The following are equivalent:
\begin{enumerate}
\item $\{P\varphi_i\}_{i\in I}$ spans $P\mathcal{H}_N$
\item $\{(I-P)\varphi_i^*\}_{i\in I^c}$ is independent.
\end{enumerate}
\end{proposition}

\begin{proof}
(1) $\Rightarrow $ (2) (Proof by contrapositive.) Assume 
$\{(I-P)\varphi_i^*\}_{i\in I^c}$ is NOT independent. Choose $\{b_i\}_{i\in I^c}$ not all zero so that $\sum_{i \in I^c}b_i(I-P)\varphi_i^* =0$. Then $x:= \sum_{i \in I^c}b_i\varphi_i^* = \sum_{i \in I^c} b_iP\varphi_i^* \in P\mathcal{H}_N$. If $j \in I$, then 
\[\langle x, P\varphi_j\rangle = \langle Px, \varphi_j\rangle = \langle x, \varphi_j\rangle = \sum _{i \in I^c} b_i\langle \varphi_i^*, \varphi_j \rangle =0\] since $i \in I^c$ and $j \in I$. Thus $x \perp \lspan \{P\varphi_j\}_{j\in I}$ showing that $\{P\varphi_j\}_{j\in I}$ does not span $P\mathcal{H}_N$.

(2) $\Rightarrow$ (1) (Proof by contrapositive.) Assume $\lspan \{P\varphi_i\}_{i\in I} \neq P\mathcal{H}_N$. Then there exists a non-zero $x\in P\mathcal{H}_N$ with $x \perp \lspan \{P\varphi_i\}_{i\in I}$. Also $x = \sum_{i=1}^N \langle x,\varphi_i\rangle \varphi_i^*$. Now for $i \in I, \langle x, P\varphi_i\rangle = \langle Px, \varphi_i \rangle = \langle x, \varphi_i \rangle = 0$. Hence 
\[\sum_{i\in I^c}\langle x, \varphi_i \rangle \varphi_i^*=x=Px=\sum_{i \in I^c}\langle x,\varphi_i\rangle P\varphi_i^*.\] Thus, $\sum_{i\in I^c} \langle x, \varphi_i \rangle (I-P)\varphi_i^* =0$ where $\langle x, \varphi_i \rangle \neq 0$ for at least one $i \in I^c$. Therefore $\{(I-P)\varphi_i^*\}_{i\in I^c}$ is NOT independent.\end{proof}

Similarly, the following result classifies full spark projections of
Riesz bases as follows.

\begin{proposition}\label{propfsrb}
 Let $\{\varphi_i\}_{i=1}^N$ be a Riesz basis on $\mathcal{H}_N$ and $P$ an orthogonal projection on $\mathcal{H}_N$ of rank $M$. The following are equivalent:
\begin{enumerate}
\item  $\{P\varphi_i\}_{i=1}^N$ is full spark.

\item For all $I\subset \{1,2,\ldots,N\}$ with $|I|=M$ we have:
\[ \left [ \lspan\{(I-P)\varphi_i\}_{i=1}^N \right ] \cap \lspan \{\varphi_i\}_{i\in I} = \{0\}.\]
\end{enumerate}
\end{proposition}

\begin{proof}
$\{P\varphi_i\}_{i=1}^N$ is not full spark if and only if there exists an $I\subset \{1,2,\ldots,N\}$ with $|I|=M$ and $\{a_i\}_{i \in I}$ not all zero such that $\sum_{i\in I}a_iP\varphi_i = 0$ if and only if $\sum_{i\in I}a_i(I-P)\varphi_i=\sum_{i\in I}a_i\varphi_i$ if and only if 
\[\left[\lspan  \{(I-P)\varphi_i\}_{i=1}^N \right ] \cap \lspan \{\varphi_i\}_{i\in I} \neq \{0\}.\] \end{proof}

A natural question to ask, in light of Proposition \ref{propfsrb}, is if $\{\varphi_i\}_{i=1}^N$ is a Riesz basis for $\mathcal{H}_N$ and $\{P\varphi_i\}_{i=1}^N$ is full spark on its range for some rank-$M$ projection $P$, then is $\{(I-P)\varphi_i\}_{i=1}^N$ full spark on its range? The following example shows that the answer is no.

\begin{example}
Let $\{e_1, e_1+e_2\}$ be a Riesz basis for $\mathcal{R}^2$, where $\{e_i\}_{i=1}^2$ is the standard orthonormal basis for $\mathcal{R}^2$. Let $P$ be the rank-1 projection onto $e_1$. Then $\{Pe_1,P(e_1+e_2)\}=\{e_1\}$ is full spark on its range. However, $\{(I-P)e_1, (I-P)(e_1+e_2)\}=\{0, e_2\}$ is not full spark on its range. 
\end{example}

\section{Phase Retrieval}
Phase retrieval has been a long standing problem in mathematics and engineering alike. Recently, the mathematical study of phase retrieval by subspace components (or projections) has been more deeply developed and we further study this area here. In particular, this section expands and generalizes a few results from \cite{CCPW} and \cite{CCJW} as well as develops many new results with its main theorem focused on classifying phase retrieval by projections via norm retrieval by vectors. 

\begin{definition}\label{D1}
A family of vectors $\{\varphi_i\}_{i=1}^M$ does {\bf phase retrieval} in $\mathcal{H}$ if whenever $x,y \in \mathcal{H}$ satisfy
\[ |\langle x,\varphi_i\rangle|=|\langle y,\varphi_i\rangle|,
\mbox{ for all }i=1,2,\ldots,M,\]
then $x=cy$ for some $|c|=1$.

A family of projections $\{P_i\}_{i=1}^M$ does {\bf phase retrieval} in $\mathcal{H}$ if whenever $x,y\in \mathcal{H}$ satisfy
\[ \|P_ix\|=\|P_iy\|,\mbox{ for all }i=1,2,\ldots,M,\]
then $x=cy$ for some $|c|=1$.

If $\{W_i\}_{i=1}^M$ is a family of subspaces of $\mathcal{H}$, we say
they do {\bf phase retrieval} if the orthogonal projections onto the $W_i$ do phase retrieval.
\end{definition}

One of the main results in \cite{Balan06} is the following.

\begin{theorem}
A family of vectors $\Phi$ in $R^N$ does phase retrieval if and only if $\Phi$ has the complement property.  In the complex case, phase retrieval implies complement property
but the converse fails.
\end{theorem}

In \cite{CCPW}, the authors analyzed phase retrieval by projections and provided the following results. 

\begin{theorem} 
We have:
\begin{enumerate}
\item Phase retrieval in $\mathcal{R}^N$ is possible using $2N-1$ subspaces, each of any dimension less than $N$.

\item Phase retrieval in $\mathcal{C}^N$ is possible using $4N-3$ subspaces, each of any dimension less than $N$.
\end{enumerate}
\end{theorem}

Also, recall the following result from \cite{CCJW}.
\begin{proposition}\label{prop5.4}
If $\{\varphi_i\}_{i=1}^M$ is a frame in $\mathcal{H}_N$ which allows phase retrieval then $\{P\varphi_i\}_{i=1}^M$ allows phase retrieval for all orthogonal projections $P$ on $\mathcal{H}_N$.
\end{proposition}

Generalizing Proposition \ref{prop5.4} to subspaces leads to the following result.

\begin{proposition}\label{onb}
If $\{e_i\}_{i=1}^{2N}$ is an orthonormal basis for $\mathcal{H}_{2N}$, then for all $k \leq N$ there exists a projection of $\mathcal{H}_{2N}$ onto a $k$-dimensional subspace $W$ such that $\{Pe_i\}_{i=1}^{2N}$ does phase retrieval on its range.
\end{proposition}

\begin{proof}
Without loss of generality assume $k=N$, choose any Parseval frame $\{g_i\}_{i=1}^{2N}$ which does phase retrieval for an $N$-dimensional Hilbert space $K_N$. Let $T$ be its analysis operator, then $T$ is an isometry and $T:K_N \rightarrow \ell_2^{2N}$ with $Tg_i=Pe_i$ for all $i=1,\dots, 2N$. Thus $P$ is a projection onto the range of $T$, a $k$-dimensional space, and $\{Pe_i\}_{i=1}^{2N}=\{Tg_i\}_{i=1}^{2N}$ gives phase retrieval.\end{proof}

We can further generalize Proposition \ref{onb} to Riesz bases as follows.

\begin{proposition}\label{proprb}
If $\Phi=\{\varphi_i\}_{i=1}^M$ is a Riesz basis for $\mathcal{H}_M$ and $N\leq M$ then there exists a projection $P$ on $\mathcal{H}_M$ of rank $N$ so that $\{P\varphi_i\}_{i=1}^M$ is full spark on its range. Hence, if $2N-1 \leq M$ then $\{P\varphi_i\}_{i=1}^M$ yields phase retrieval on its range.
\end{proposition}

\begin{proof}
Let $S$ be the frame operator for $\Phi$. Then $\{S^{-\frac{1}{2}}\varphi_i\}_{i=1}^M$ is an orthonormal basis for $\mathcal{H}_M$. Hence, there exists a projection $P$ of rank $N$ such that $\{PS^{-\frac{1}{2}}\varphi_i\}_{i=1}^M$ is full spark. Thus by Proposition \ref{propfsrb}, $(I-P)\mathcal{H}_M \cap \lspan\{S^{-\frac{1}{2}}\varphi_i\}_{i\in I} = \{0\}$ for all $I\subset \{1,\dots, M\}$ with $|I|=N$. 

Let $W=S^{\frac{1}{2}}(I-P)\mathcal{H}_M$ and suppose $x \in W \cap \lspan\{\varphi_i\}_{i\in I}$. Then $x = S^{\frac{1}{2}}(I-P)y=\sum_{i\in I}b_i\varphi_i$ for some $ y\in \mathcal{H}_M$, and $b_i\in H$. So 
\[S^{-\frac{1}{2}}x=(I-P)y=\sum_{i\in I}b_iS^{-\frac{1}{2}}\varphi_i \in (I-P)\mathcal{H}_M \cap \lspan\{S^{-\frac{1}{2}}\varphi_i\}_{i\in I}.\] 
Hence \[W \cap \lspan\{\varphi_i\}_{i\in I}=S^{\frac{1}{2}}(I-P)\mathcal{H}_M\cap \lspan\{\varphi_i\}_{i\in I}=\{0\}.\]
Now let $Q$ be the projection onto $W$. Then $\{(I-Q)\varphi_i\}_{i=1}^M = \{P\varphi_i\}_{i=1}^M$ is full spark, by Proposition \ref{propfsrb}.
When $2N-1\leq M$ then $\{(I-Q)\varphi_i\}_{i=1}^M = \{P\varphi_i\}_{i=1}^M$ is full spark and thus has the complement property. Therefore $\{P\varphi_i\}_{i=1}^M$ yields phase retrieval. \end{proof}

\begin{remark}
The set of vectors $\{e_1,e_2,e_3,e_1+e_2,e_1+e_3,e_2+e_3\}$ have the property that it does phase retrieval but no full spark subset does phase retrieval.
\end{remark}

\begin{corollary}
If $\{\varphi_i\}_{i=1}^M$ is a Riesz basis for $\mathcal{H}_M$ with dual Riesz basis $\{\varphi_i^*\}_{i=1}^M$ and $2N-1\leq M \leq 2N+1$, then there exists a projection $P$ of rank $N$ so that both of the following hold:
\begin{enumerate}
\item $\{P\varphi_i\}_{i=1}^M$ does phase retrieval on its range, and
\item $\{(I-P)\varphi_i^*\}_{i=1}^M$ does phase retrieval on its range. 
\end{enumerate}
\end{corollary}

\begin{proof}
The proof follows from Proposition \ref{proprb}.
\end{proof}

In \cite{CCPW}, the authors provided necessary and sufficient conditions for when subspaces do phase retrieval by relating them to the one dimensional case.  Their results were proven for the real case.  The complex case is more
technical so we will prove this here. To accomplish this, we give a few preliminary results. Note that the following results hold in both the real and complex case.

\begin{lemma}\label{pete1}
Let $\{W_i\}_{i=1}^M$ be subspaces of $\mathcal{H}_{N}$ allowing
phase retrieval.  For every orthonormal basis $\{\varphi_{i,j}\}_{j=1}^{J_i}$ of $W_i$, the set $\Phi=\{\varphi_{i,j}\}_{i=1,j=1}^{M,J_i}$ allows phase retrieval in $\mathcal{H}_N$.
\end{lemma}

\begin{proof}
Let $P_i$ be the orthogonal projection onto $W_i$ for each $i=\{1,\dots, M\}$. For every $i=1,2,\ldots,M$ and for any $x,y \in \mathcal{H}_N$ such that $|\langle x,\varphi_{i,j}\rangle|= |\langle y, \varphi_{i,j}\rangle|$ for all $j=\{1,\dots, J_i\}$, we have
\[ \|P_ix\|^2 = \sum_{j=1}^{J_i}|\langle x,\varphi_{i,j} \rangle|^2
= \sum_{j=1}^{J_i}|\langle y,\varphi_{i,j} \rangle|^2
= \|P_iy\|^2.\]
Since $\{W_i\}_{i=1}^M$ allows phase retrieval it follows
that $x=cy$ for some $c\in \mathcal{C}$ with $|c|=1$. Hence, $\Phi$ allows phase retrieval.
\end{proof}

\begin{lemma}\label{pete2}
Let $P$ be a projection onto an $M$-dimensional subspace, $W$,
of $\mathcal{H}_N$.  Given $x,y \in \mathcal{H}_N$, the following are
equivalent:
\begin{enumerate}
\item $\|Px\|=\|Py\|$.
\item There exists an orthonormal basis 
$\{\varphi_i\}_{i=1}^M$ for $W$ such that $|\langle x,
\varphi_i\rangle|=|\langle y,\varphi_i\rangle|$, for all
$i=1,2,\ldots,M$.
\end{enumerate} 
\end{lemma}

\begin{proof}
$(1) \Rightarrow (2)$:  Consider the vectors $Px,Py \in W$
with $\|Px\|=\|Py\|$. We examine the following three cases.
\vskip12pt
\noindent {\bf Case 1}:  Assume $Px = cPy$ for some
$|c|=1$.
\vskip12pt
In this case, for any orthonormal basis $\{\varphi_i\}_{i=1}^M$ for $W$ we have \[|\langle x, \varphi_i\rangle|=|\langle x, P\varphi_i\rangle|= |\langle Px, \varphi_i\rangle|=|\langle cPy,\varphi_i\rangle|=|c||\langle y,P\varphi_i\rangle|=|\langle y,\varphi_i\rangle|,\] for all
$i=1,2,\ldots,M$ as desired.

\

Hence, for the next two cases, we can assume $Px \not= cPy$ for
any $|c|=1$.

\vskip12pt
\noindent {\bf Case 2}:  Assume $\langle Px,Py\rangle =0$. Hence $\langle Py,Px \rangle = 0$.
\vskip12pt
In this case, let 
\[ \varphi_1 = \frac{Px+Py}{\|Px+Py\|},\mbox{ and }
\varphi_2 = \frac{Px-Py}{\|Px-Py\|}.\]

Note $\varphi_1$ and $\varphi_2$ are both unit norm. 
Letting $c= 1/(\|Px+Py\|\|Px-Py\|)$ we have
\[ \langle \varphi_1,\varphi_2\rangle = c\langle 
Px+Py,Px-Py\rangle = \|Px\|^2 -\|Py\|^2 +\langle Py,Px\rangle
- \langle Px,Py\rangle =0.\]
So $\{\varphi_1,\varphi_2\}$ is an orthonormal set.  
Also,
\begin{align*}
 |\langle x,Px+Py\rangle| &= |\langle Px,Px\rangle +\langle Px,Py\rangle| = \|Px\|^2\\
&= \|Py\|^2 = |\langle Py,Px\rangle +\langle Py,Py\rangle|= |\langle y,Px+Py\rangle|. \end{align*}

Similarly, $|\langle x,Px-Py\rangle| = |\langle y,Px-Py \rangle|.$
Hence, $|\langle x, \varphi_i\rangle|=|\langle y, \varphi_i\rangle|$ for $i=1,2$.
Note that $Px, Py \in \lspan {\{\varphi_1, \varphi_2\}}$. Now, take $\{\varphi_i\}_{i=1}^M$ to be any orthonormal completion of $\{\varphi_1,\varphi_2\}$ to an orthonormal basis for
$W$. Then, for all $i=3,4,\ldots, M$ we have
\[ \langle x,\varphi_i\rangle = \langle x, P\varphi_i\rangle = \langle Px, \varphi_i\rangle= 0 =\langle Py,\varphi_i\rangle = \langle y, P\varphi_i\rangle = \langle y,\varphi_i\rangle.\]
Therefore, $|\langle x, \varphi_i\rangle|=|\langle y,\varphi_i\rangle|$ for all $i=\{1,\dots, M\}$ where $\{\varphi_i\}_{i=1}^M$ is an orthonormal basis for $W$, as desired. 

\vskip12pt
\noindent {\bf Case 3}:  $\langle Px,Py\rangle \not= 0$.
\vskip12pt
In this case, let $d = \langle Px,Py\rangle /|\langle 
Px,Py\rangle |$ so that $|d|=1$, and let
\[ \varphi_1 = \frac{Px+dPy}{\|Px+dPy\|},\mbox{ and }
\varphi_2 = \frac{Px-dPy}{\|Px-dPy\|}.\]
Note that $\varphi_1$ and $\varphi_2$ are both unit norm. 
Letting $c = 1/\|Px+dPy\|\|Px-dPy\|$ we have
\begin{eqnarray*}
\langle \varphi_1,\varphi_2\rangle =&c\langle Px+dPy,
Px-dPy\rangle\\
=& c(\|Px\|^2 -\|dPy\|^2 + \langle dPy,Px\rangle -
\langle Px,dPy\rangle)\\
=& c\Big( \|Px\|^2 -|d|\|Py\|^2) + d\langle Py,Px\rangle - \overline{d}\langle Px,Py\rangle \Big)\\
=& \displaystyle c\Big((1-|d|)||Px||^2 + \frac{\langle Px,Py\rangle \langle Py,Px\rangle} {|\langle Px,Py\rangle |} - \frac{ \overline{\langle Px,Py\rangle}  \langle Px,Py\rangle} {|\langle Px,Py\rangle | }\Big)\\
=& \displaystyle c\Big( 0+\frac{|\langle Px,Py\rangle|^2} {|\langle Px,Py\rangle |} - \frac{ |\langle Px,Py\rangle|^2}{|\langle Px,Py\rangle | }\Big)\\
=&0.
\end{eqnarray*} 

Hence, $\{\varphi_1,\varphi_2\}$ is an orthonormal set.
Now the proof follows as in Case 2.  

\vskip12pt
$(2) \Rightarrow (1)$:  This is immediate by:
\[ \|Px\|^2 = \sum_{i=1}^M|\langle x,\varphi_i \rangle|^2
= \sum_{i=1}^M|\langle y,\varphi_i\rangle|^2 = \|Py\|^2.\]\end{proof}

Combining Lemma \ref{pete1} and Lemma \ref{pete2}, we arrive at a characterization for when $\{W_i\}_{i=1}^M$ does phaseless reconstruction in $\mathcal{H}_N$ in terms of orthonormal bases. Note, in \cite{CCPW} the authors proved the following result for the real case.  We will now give the proof for the complex case.

\begin{theorem}\label{CCPW} 
Let $\{W_i\}_{i=1}^M$ be subspaces of $\mathcal{H}_{N}$. The following are equivalent:
\begin{enumerate}
\item $\{W_i\}_{i=1}^M$ allows phase retrieval in $\mathcal{H}_N$.
\item For every orthonormal basis $\{\varphi_{i,j}\}_{j=1}^{J_i}$ of $W_i$, the set $\{\varphi_{i,j}\}_{i=1,j=1}^{M,J_i}$ allows phase retrieval in $\mathcal{H}_N$.
\end{enumerate}
\end{theorem}

\begin{proof}
$(1) \Rightarrow (2)$:  This is Lemma \ref{pete1}.

$(2) \Rightarrow (1)$:  Suppose we have $x,y \in \mathcal{H}_N$ with
$\|P_ix\|=\|P_iy\|$, for all $i=1,2,\ldots,M$.  By Lemma
\ref{pete2} we can choose orthonormal bases $\Phi =\{\varphi_{i,j}
\}_{i=1,j=1}^{M, J_i}$ so that 
\[ |\langle x,\varphi_{i,j}\rangle |=|\langle y,\varphi_{i,j}
\rangle|,\mbox{ for all } i,j.\]
By (2), $\Phi$ does phase retrieval and so $x=cy$ for some
$|c|=1$.  I.e.  $\{W_i\}_{i=1}^M$ does phase retrieval. 
\end{proof}

Since orthonormal bases are very restrictive, we would like to relax the conditions in Theorem \ref{CCPW} to see what properties the vectors within the subspaces have when the $\{W_i\}_{i=1}^M$ are assumed to allow phase retrieval. A natural next step would be to look at a Riesz basis as opposed to orthogonal vectors. In particular, since unitary operators are the only linear operators which preserve orthogonality, by moving to a Riesz basis we can instead look at invertible operators. However, the following example shows that if $\{W_i\}_{i=1}^M$ allows phase retrieval in $\mathcal{R}^N$ and $\{\varphi_{ij}\}_{j=1}^{J_i}$ is a Riesz basis for $W_i$ for each $i \in \{1,\dots,M\}$ then it is not necessarily true that $\{\varphi_{ij}\}_{j=1,i=1}^{M,J_i}$ allows phase retrieval in $\mathcal{R}^N$.

\begin{example} \label{JohnsEx}
\rm Let $\{e_i\}_{i=1}^3$ be an orthonormal basis for $\mathcal{R}^3$. Define the subspaces
\[W_1=\lspan\{e_1, e_2\}, W_2= \lspan\{e_2\}, W_3= \lspan\{e_3\},\]
\[ W_4= \lspan\{\frac{e_1+e_2}{2}\}, W_5= \lspan\{\frac{e_2+e_3}{2}\}, W_6= \lspan\{\frac{e_1+e_3}{2}\}.\]

Let $x \in R^3$. Then $x=\sum_{i=1}^3 \alpha_i e_i$ where $\alpha_i=\langle x,e_i\rangle$ for $i=1,2,3$.
We have $$||P_{W_i}x||^2=\left\{
\begin{array}{cc}
\alpha_1^2 + \alpha_2^2, & i=1\\
 \alpha_2^2, & i=2\\
\alpha_3^2, & i=3\\
 \frac{1}{2}\left( \alpha_1 + \alpha_2 \right)^2, & i=4\\
\frac{1}{2}\left( \alpha_2 + \alpha_3 \right)^2, & i=5\\
\frac{1}{2}\left( \alpha_1 + \alpha_3 \right)^2, & i=6
\end{array} \right.
$$

First, we will show that we can recover $\pm x$ from $\{||P_{W_i}x||^2\}_{i=1}^6$.

We can recover the absolute values of the coefficients:
\[|\alpha_1|=\sqrt{||P_{W_1}x||^2-||P_{W_2}x||^2}, |\alpha_2|=||P_{W_2}x||, |\alpha_3|=||P_{W_3}x||.\]

Thus, if two of the coefficients are zero then we have $x=\pm|\alpha_i|e_i$ for some $i$. From now on we will assume at least two of the coefficients are nonzero.

Case 1: Assume $\alpha_1=0$. We may assume without loss of generality that $\alpha_2 >0$, and thus $\alpha_2=||P_{W_2}x||$. Finally,
\[\alpha_3=\frac{2||P_{W_5}x||^2-||P_{W_2}x||^2-||P_{W_3}x||^2}{2||P_{W_2}x||}.\]

Case 2: Assume $\alpha_1 \neq 0$. We may assume without loss of generality that $\alpha_1>0$, and thus
\[\alpha_1 = \sqrt{||P_{W_1}x||^2-||P_{W_2}x||^2}.\]

We have
\[\alpha_2=\frac{2||P_{W_4}x||^2-||P_{W_1}x||^2}{2\sqrt{||P_{W_1}x||^2-||P_{W_2}x||^2}}\]

and
\[\alpha_3=\frac{2||P_{W_6}x||^2+||P_{W_2}x||^2-||P_{W_1}x||^2-||P_{W_3}x||^2}{2\sqrt{||P_{W_1}x||^2-||P_{W_2}x||^2}}.\]

This shows that $\{W_i\}_{i=1}^6$ does phase retrieval.

However, if we choose the linearly independent (and not orthonormal) basis $\{e_1+e_2,e_2\}$ for $W_1$ and the spanning element from the other subspaces, we get the set of vectors $\{e_1+e_2,e_2,e_2,e_3,e_1+e_2,e_2+e_3, e_1+e_3\}=\{e_1+e_2,e_2,e_3,e_2+e_3, e_1+e_3\}$. Notice that if we partition this set as follows $\{e_2+e_3,e_2,e_3\},\{e_1+e_2, e_1+e_3\}$ then neither set spans $R^3$. Hence this set does not have the complement property and therefore does not allow phase retrieval. 
\end{example}

\begin{remark} 
In Example \ref{JohnsEx}, $\{W_i^{\perp}\}_{i=1}^6$ does phase retrieval.  To see this, we just need to see that this
family can retrieve the norm of any vector $x \in \mathcal{R}^3$.  Let $\{Q_i\}_{i=1}^6$ be the orthogonal projections onto each of $\{W_i^{\perp}\}_{i=1}^6$.  Then,
\[ W_1^{\perp}= \lspan\{e_3\},\quad W_2^{\perp}=\lspan\{e_1,e_3\}\quad W_3^{\perp}=\lspan\{e_1,e_2\}.\]
Given a vector $x \in \mathcal{R}^3$, we have
\[ \|Q_1x\|^2 =|\langle x,e_3\rangle|^2\mbox{ and } \|Q_2x\|^2 = |\langle x,e_1\rangle |^2 + |\langle x,e_3\rangle|^2\]
\[\mbox{ and }
\|Q_3x\|^2 = |\langle x,e_1\rangle|^2+|\langle x,e_2\rangle|^2.\]
 Hence we know $\{|\langle x,e_i\rangle|^2\}_{i=1}^3$ and so we know $\|x\|^2$.
\end{remark}

In light of Example \ref{JohnsEx} we cannot replace ``orthonormal bases'' with ``Riesz bases'' in Theorem \ref{CCPW} (2). The next theorem shows that the key property of orthonormal bases in Theorem \ref{CCPW} is not just that they span, but that they give {\it norm retrieval}.

\begin{definition}
A family of vectors $\{\varphi_i\}_{i=1}^M$ does {\bf norm retrieval} in $\mathcal{H}_N$ if whenever $x,y \in \mathcal{H}_N$ satisfy
\[ |\langle x,\varphi_i\rangle|=|\langle y,\varphi_i\rangle|
\mbox{ for all }i=1,2,\ldots,M,\]
then $\|x\|=\|y\|$.

A family of projections $\{P_i\}_{i=1}^M$ does {\bf norm retrieval} in $\mathcal{H}_N$ if whenever $x,y\in \mathcal{H}_N$ satisfy
\[ \|P_ix\|=\|P_iy\| \mbox{ for all }i=1,2,\ldots,M,\]
then $\|x\|=\|y\|$.
\end{definition}

\begin{theorem}\label{pn} Let $\{W_{i}\}_{i=1}^{M}$ be subspaces of $\mathcal{H}_{N}$. The following are equivalent:
\begin{enumerate}
\item $\{W_{i}\}_{i=1}^{M}$ allows phase retrieval in $\mathcal{H}_{N}$.
\item For every sequence $\{\varphi_{i,j}\}_{j=1}^{J_{i}}\subset W_{i}$ which gives norm retrieval in $W_{i}$, the sequence $\{\varphi_{i,j}\}_{j=1,i=1}^{J_{i},M}$ allows phase retrieval.
\end{enumerate}

\begin{proof} (1) $\Rightarrow$ (2) For each $i=1,\ldots,M$ let $\{\varphi_{i,j}\}_{j=1}^{J_{i}}$ be a sequence in $W_{i}$ which gives norm retrieval in $W_{i}$. Let $x,y\in \mathcal{H}_{N}$ such that $|\langle x,\varphi_{i,j}\rangle|=|\langle y,\varphi_{i,j}\rangle|$ for all $j=1,\ldots,J_{i}$, $i=1,\ldots,M$. For each $i=1,\ldots,M$ let $P_{i}$ be the projection onto $W_{i}$. We have
\[|\langle P_{i}x,\varphi_{i,j}\rangle| = |\langle x,P_{i}\varphi_{i,j}\rangle|=|\langle x,\varphi_{i,j}\rangle|=|\langle y,\varphi_{i,j}\rangle| = |\langle y,P_{i}\varphi_{i,j}\rangle| = |\langle P_{i}y,\varphi_{i,j}\rangle|,\] 
and since $\{\varphi_{i,j}\}_{j=1}^{J_{i}}$ gives norm retrieval in $W_{i}$, we have $\norm{P_{i}x}=\norm{P_{i}y}$ for all $i=1,\ldots,M$. Since $\{W_{i}\}_{i=1}^{M}$ gives phase retrieval, we have $x=cy$ for some $c$ with $|c|=1$.

(2) $\Rightarrow$ (1) Since orthonormal bases give norm retrieval, (2) implies that each sequence $\{\varphi_{i,j}\}_{j=1,i=1}^{J_{i},M}$ gives phase retrieval, where $\{\varphi_{i,j}\}_{j=1}^{J_{i}}$ is an orthonormal basis for $W_{i}$. Theorem \ref{CCPW} implies that the subspaces $\{W_{i}\}_{i=1}^{M}$ allow phase retrieval.\end{proof}

\end{theorem}

Theorem \ref{pn} provides a clear connection between phase retrieval by projections and sequences of vectors giving norm retrieval. In conjunction with this theorem, recall the following theorems from \cite{CCPW} and \cite{CCJW}.

\begin{theorem}\label{aaa}\cite{CCJW}
Given $\{\varphi_i\}_{i=1}^M$ spanning $\mathcal{H}_N$, there exists an invertible operator $T$ on $\mathcal{H}_N$ so that the collection of orthogonal projections onto the vectors $\{T\varphi_i\}_{i=1}^M$ allows norm retrieval.
\end{theorem}

\begin{theorem}\label{bbb}\cite{CCPW}
If a family of vectors $\{\varphi_i\}_{i=1}^M$ does phase
retrieval on $\mathcal{H}_N$ and $T$ is an invertible operator on
$\mathcal{H}_N$, then $\{T\varphi_i\}_{i=1}^M$ does phase retrieval.
\end{theorem}

In light of these results, it seems natural to question if invertible operators preserve phase retrieval by projections. However, the following example illustrates that Theorem \ref{bbb} fails in the higher dimensional case.

\begin{example}\label{JohnsEx2}
Let $\{e_i\}_{i=1}^3$ be the standard orthonormal basis in $\mathcal{R}^3$. Define the subspaces $\{W_i\}_{i=1}^6$ as in Example \ref{JohnsEx}. Define the linear operator $T$ on the basis elements by:

\[Te_i=
\begin{cases}
e_1-e_2 & i=1\\
e_2 & i=2\\
e_3 & i=3
\end{cases}\]

We have
\[TW_1=W_1, TW_2=W_2, TW_3=W_3,\]
\[ TW_4=\lspan \{e_1\}, TW_5=W_5, TW_6=\lspan \{e_1-e_2+e_3\}.\]

We can choose the following orthonormal bases from each subspace:
\[\{e_1,e_2\} \subset TW_1, \{e_2\}\subset TW_2, \{e_3\} \subset TW_3, \{e_1\} \subset TW_4,\] \[\{\frac{1}{\sqrt{2}}\left(e_2+e_3\right)\} \subset TW_5, \{\frac{1}{\sqrt{3}}\left(e_1-e_2+e_3\right)\}\subset TW_6.\]

We claim that the union of these bases \[\{e_1,e_2,e_3,\frac{1}{\sqrt{2}}\left(e_2+e_3\right), \frac{1}{\sqrt{3}}\left(e_1-e_2+e_3\right)\}\] does not yield phase retrieval.

Indeed, if we partition this set into the following two sets:
\[\{e_2,e_3,\frac{1}{\sqrt{2}}\left(e_2+e_3\right)\}, \{e_1,\frac{1}{\sqrt{3}}\left(e_1-e_2+e_3\right)\}\]
then we see that neither of the two sets spans $\mathcal{R}^3$. Thus the subspaces $\{TW_i\}_{i=1}^6$ do not yield phase retrieval. Therefore, we have shown that there exists a bounded operator $T$ such that $\{TW_i\}_{i=1}^6$ does not yield phase retrieval, even though the original subspaces, $\{W_i\}_{i=1}^6$, did yield phase retrieval.
\end{example}

\section{Subspaces which allow phase retrieval}

In this section we develop properties of vectors and subspaces which allow phase retrieval. In particular, we start with a frame which fails phase retrieval and compute the minimum number of vectors necessary to add to the frame so that it will yield phase retrieval. After this computation, we then actually construct such sets. The following results show that for a given frame which fails the complement property, then for all partitions $I_1,I_2$ of the frame for which neither set spans $\mathcal{R}^M$, there is a minimal number of vectors we can add to make this set possess the complement property. Moreover, there is an open dense set of vectors which suffice.

\begin{theorem}\label{thmdef}
Let $\{\varphi_i\}_{i=1}^M$ be a frame for $\mathcal{R}^N$. Fix $k \in N \cup \{0\}$ such that
\[N-k=\min\limits_{\substack{I\subset \{1,\dots,M\}\\[2pt] \lspan \{\varphi_i\}_{i\in I}\neq \mathcal{R}^N}}\left(\dimm(\lspan \{\varphi_i\}_{i\in I^c}\right)).\]

Then there exists $\{\varphi_{M+1},\dots,\varphi_{M+k}\}\subset \mathcal{R}^N$ such that $\{\varphi_i\}_{i=1}^{M+k}$ has the complement property. Moreover, if $\{\varphi_i\}_{i=1}^M \cup \{\psi_j\}_{j=1}^s$ has the complement property, then $s \geq k$.
\end{theorem}

\begin{proof}
Let $\{\varphi_i\}_{i=1}^M$ be a frame for $\mathcal{R}^N$. Fix $k \in M \cup \{0\}$ such that
\[N-k=\min\limits_{\substack{I\subset \{1,\dots,M\}\\[2pt] \lspan \{\varphi_i\}_{i\in I}\neq \mathcal{R}^N}}\left( \dimm(\lspan \{\varphi_i\}_{i\in I^c}\right)).\]

Define

\[A_0:= \{I\colon \dimm(\lspan \{\varphi_i\}_{i\in I}) <N \mbox{ and }\dimm(\lspan \{\varphi_i\}_{i\in I^c}) <N\}.\]

Since $\{\varphi_i\}_{i=1}^M$ is a finite collection of vectors then there is a finite number of sets $I \subset A_0$. For each $I\subset A_0$ we see that $\dimm(\lspan \{\varphi_i\}_{i\in I}) \leq N-1$ and $\dimm(\lspan \{\varphi_i\}_{i\in I^c}) \leq N-1$, and hence $\lspan \{\varphi_i\}_{i\in I}$ and $\lspan \{\varphi_i\}_{i\in I^c}$ both have measure zero. This implies that there is a vector (in fact, an open dense set of vectors) $\varphi_{M+1}$ such that $\varphi_{M+1} \notin \lspan \{\varphi_i\}_{i\in I}$, $\varphi_{M+1} \notin \lspan \{\varphi_i\}_{i\in I^c}$, $\dimm(\lspan (\{\varphi_i\}_{i\in I}\cup \varphi_{M+1})) \geq N-k+1$ and $\dimm(\lspan (\{\varphi_i\}_{i\in I^c}\cup \varphi_{M+1})) \geq N-k+1$, for any $I\subset A_0$.

Next, define 
\[A_1:= \{I\colon \dimm(\lspan  (\{\varphi_i\}_{i\in I}\cup \varphi_{M+1})) <N \mbox{ and }\dimm(\lspan  (\{\varphi_i\}_{i\in I^c}\cup \varphi_{M+1})) <N\}.\]
For each $I\subset A_1$, $\dimm(\lspan(\{\varphi_i\}_{i\in I} \cup \varphi_{M+1})) \leq N-1$ and $\dimm(\lspan(\{\varphi_i\}_{i\in I^c}\cup \varphi_{M+1})) \leq N-1$ and hence have measure zero. Thus the union of all subspaces defined by $I \subset A_1$ have measure zero. Hence there exists a vector $\varphi_{M+2}$ such that $\varphi_{M+2} \notin \lspan(\{\varphi_i\}_{i\in I}\cup \varphi_{M+1})$, $\varphi_{M+2} \notin \lspan(\{\varphi_i\}_{i\in I^c}\cup \varphi_{M+1})$, $\dimm(\lspan(\{\varphi_i\}_{i\in I}\cup \{\varphi_{M+j}\}_{j=1}^2)) \geq N-k+2$ and $\dimm(\lspan(\{\varphi_i\}_{i\in I^c}\cup \{\varphi_{M+j}\}_{j=1}^2)) \geq N-k+2$, for any $I\subset A_1$.

Continue in this manner by defining $A_k$ and adding $\varphi_{M+k}$ until $|A_k|=0$. Therefore $\{\varphi_i\}_{i=1}^{M+k}$ has the complement property.

For the \emph{moreover} part let $I\subset\{1,\ldots,M\}$ be such that $\lspan \{\varphi_{i}\}_{i\in I}\neq \mathcal{R}^N$ and
\[N-k=\dimm(\lspan \{\varphi_i\}_{i\in I^c}).\]
If $\{\psi_{j}\}_{j=1}^{s}$ is a sequence such that $\{\varphi_{i}\}_{i=1}^{M}\cup\{\psi_{i}\}_{j=1}^{s}$ has the complement property, then $\{\varphi_{i}\}_{i\in I^{c}}\cup\{\psi_{i}\}_{j=1}^{s}$ must span $\mathcal{R}^N$, that is
\[N=\dimm(\lspan (\{\varphi_{i}\}_{i\in I^{c}}\cup\{\psi_{j}\}_{j=1}^{s}))\leq N-k+s.\]
\end{proof}

\begin{theorem}\label{dense}
Let $\{\varphi_i\}_{i=1}^M$ be a frame for $\mathcal{R}^N$. Fix $k \in N \cup \{0\}$ such that
\[N-k=\min\limits_{\substack{I\in \{1,\dots,M\}\\[2pt] \lspan\{\varphi_i\}_{i\in I}\neq \mathcal{R}^N}}\left(\dimm(\lspan \{\varphi_i\}_{i\in I^c})\right).\]

Then there exists an open dense set $O_{M+1}\subset \mathcal{R}^N$ such that for all $\varphi_{M+1} \in O_{M+1}$,
\[N-k+1 = \min\limits_{\substack{I\in \{1,\dots,M+1\}\\[2pt] \lspan\{\varphi_i\}_{i\in I}\neq \mathcal{R}^N}}\left(\dimm(\lspan \{\varphi_i\}_{i\in I^c})\right).\]
\end{theorem}

\begin{proof}
This follows from Theorem \ref{thmdef} since the union of all subsets $I \subset A_0$ has measure zero. Hence its complement, call it $O_{M+1}$, is full measure and  thus open and dense. For any vector $\varphi_{M+1} \in O_{M+1}$, we have $\varphi_{M+1} \notin \lspan\{\varphi_i\}_{i\in I}$, $\varphi_{M+1} \notin \lspan\{\varphi_i\}_{i\in I^c}$, $\dimm(\lspan(\{\varphi_i\}_{i\in I}\cup \varphi_{M+1})) \geq N-k+1$ and $\dimm(\lspan(\{\varphi_i\}_{i\in I^c}\cup \varphi_{M+1})) \geq N-k+1$, for any $I\subset A_0$.
\end{proof}

In light of Theorem \ref{thmdef} and Theorem \ref{dense}, the next result is surprising.  It says that if we
have families of vectors $\Phi_1,\Phi_2$ which do phase retrieval in
two different Hilbert spaces, and we want to do phase
retrieval in the orthogonal sum of these Hilbert spaces,
we will have to add a very large number of vectors to
$\Phi_1$ and $\Phi_2$.

\begin{theorem}\label{thm5.22}
Let $\Phi_j=\{\varphi_i^j\}_{i=1}^{M_j}$ yield phase retrieval on
$\mathcal{H}_{N_j}$ for $j=1,2$.  Let $H=\mathcal{H}_{N_1}\oplus \mathcal{H}_{N_2}$.
If $\Phi = \{\varphi_i\}_{i=1}^M$ is a set of vectors in $H$
and $\Phi\cup \Phi_1 \cup \Phi_2$ does phase retrieval, then
$M\ge N_1+N_2-1$.

Moreover, there exists a $\Phi = \{\varphi_i\}_{i=1}^{N_1+N_2-1}$
so that $\Phi \cup \Phi_1 \cup \Phi_2$ does phase retrieval.
\end{theorem}

 \begin{proof}

\emph{Claim:} $M\geq N_1 + N_2 -1$

\emph{Proof of claim:} If $M \leq N_1 +N_2 -2$, then there exists a partition of $\{1,\dots, M\}$ into $\{I_1, I_2\}$ with $|I_1|\leq N_2-1$ and $|I_2| \leq N_1 - 1$. Hence $\Phi_j \cup \{\varphi_i\}_{i\in I_j}$ doesn't span for $j=1,2$.

\emph{Proof of the moreover part:}
Consider the span of all subsets of $\Phi_1 \cup \Phi_2$ which do NOT span $H$ but do span at least either $\mathcal{H}_{N_1}$ or $\mathcal{H}_{N_2}$. There exists a vector $\varphi_1$ which is not in the span of any of these subsets. Do the same thing with $\Phi_1 \cup \Phi_2 \cup \{\varphi_1\}$ and pick $\varphi_2$ not in the span of any of these subsets. Continue in this way to pick $\Phi=\{\varphi_i\}_{i=1}^{N_1+N_2-1}$.

\emph{Claim:} $\Phi \cup \Phi_1 \cup \Phi_2$ does phase retrieval. In particular, it has the complement property.

\emph{ Proof of claim:} Pick a subset $J=(J_0,J_1,J_2)$ such that $J_0 \subset \{1,\dots, N_1+N_2-1\}$, $J_1 \subset \{1,\dots, N_1\}$, and $J_2 \subset \{1,\dots, N_2\}$. Either $\{\varphi_i^j\}_{i\in J_j}$ or $\{\varphi_i^j\}_{i\in J_j^c}$ spans $H_{N_j}$ for $j=1,2$.

\emph{Case 1:} If $\{\varphi_i^j\}_{i\in J_j}$ spans $\mathcal{H}_{N_j}$ for $j=1,2$, then we are done.

\emph{Case 2:} If $\{\varphi_i^j\}_{i\in J_j^c}$ spans $\mathcal{H}_{N_j}$ for $j=1,2$, then we are done.

 \emph{Case 3:} Without loss of generality assume 
\[\{\Phi_i\}_{i\in J_0} \cup \{\Phi_i^1\}_{i\in J_1} \cup \{\Phi_i^2\}_{i \in J_2} \mbox{ spans } \mathcal{H}_{N_1}\]
\[ \mbox{and }\{\Phi_i\}_{i\in J_0^c} \cup \{\Phi_i^1\}_{i\in J_1^c} \cup \{\Phi_i^2\}_{i \in J_2^c} \mbox{ spans } \mathcal{H}_{N_2}.\] We have either $|J_0| \geq N_2$ or $|J_0^c| \geq N_1$, otherwise 
\[|J_0|+|J_0^c| \leq N_1-1 + N_2-1 =N_1+N_2-2,\] a contradiction. 

Without loss of generality, assume that $|J_0|\geq N_2$.

We need to show that $\{\Phi_i\}_{i\in J_0} \cup \{\Phi_i^1\}_{i\in J_1} \cup \{\Phi_i^2\}_{i \in J_2}$ spans $H$.

Let $J_0=\{k_1 <k_2< \dots <k_{N_2}\}$. Then 
\begin{align*}
\dimm(\lspan(\{\Phi_i^1\}_{i\in J_1}, \{\Phi_{k_i}\}_{i=1}^{N_2})) 
&=
 1 + \dimm(\lspan(\{\Phi_i^1\}_{i\in J_1}, \{\Phi_{k_i}\}_{i=1}^{N_2-1}))\\
&= 2+ \dimm(\lspan(\{\Phi_i^1\}_{i\in J_1}, \{\Phi_{k_i}\}_{i=1}^{N_2-2}))\\
&= \dots \\
&= N_2 + \dimm(\lspan\mathcal{H}_{N_1}) \\
&= \dimm(\mathcal{H}).
\end{align*}

Therefore $\Phi \cup \Phi_1 \cup \Phi_2$ does phase retrieval.
\end{proof}

\section{Properties of subspaces which fail phase retrieval}

In the one-dimensional real case, the complement property completely classifies phase retrieval. However, the complement property is not a sufficient condition for a collection of subspaces to allow phase retrieval. When the complement property fails, we will see that the corresponding partition yields two hyperplanes. 

\begin{theorem}\label{lem200}
Let $\{W_i\}_{i=1}^M$ be subspaces of $\mathcal{H}_N$ which yield phase retrieval but $\{W_i\}_{i\in I}$ fails phase retrieval for any $I\subset \{1,\dots, M\}$ with $|I|=M-1$.  If $\{f_{i,j}\}_{j=1}^{d_i}$ is an orthonormal basis for $W_i$ for all $i \in I$, and $I_1$ and $I_2$ is a partition of $\{\left(i,j\right)\}_{j=1, i\in I}^{d_i}$ so that $\{f_{i,j}\}_{\left(i,j\right)\in I_l}$ for $l=1,2$ do not span $\mathcal{H}_N$, then \[ \dim\left(\lspan \{f_{i,j}\}_{\left(i,j\right)\in I_1}\right) = \dim\left(\lspan \{f_{i,j}\}_{\left(i,j\right)\in I_2}\right)= N-1.\]
\end{theorem}

\begin{proof}
Without loss of generality, assume $I = \{1,\dots, M-1\}$. Let $K_l=\lspan \{f_{i,j}\}_{\left(i,j\right)\in I_l}$ for $l=1,2$.
\vskip12pt
\noindent {\bf Case 1:}  $K_1\neq K_2$.
\vskip12pt
By way of contradiction assume that $\dim K_1 \neq \dim K_2$. Hence for some
$n \geq 1$, $\codimm(\lspan K_1)=n$ and $\codimm(\lspan K_2)\geq n+1$. Then
 \[\dim\left( W_M \cap K_1\right) \geq \dim\left(W_M\right)-n.\]
 Set $W_M^{'}:=\{\mbox{the orthogonal complement of } W_M \cap K_1 \mbox{ inside of } W_M\}$.  So $\dimm(W'_M) \le n$.
  Next, pick an orthonormal basis $\mathcal{G}=\{g_i\}_{i=1}^m$ for $W'_M$ and an orthonormal
  basis $\mathcal{L}=\{h_i\}_{i=1}^p$ for $W_M \cap K_1$.  Note that $m\le n$.  Now, $\mathcal{G}\cup \mathcal{L}$
  is an orthonormal basis for $W_M$.  Add $\{g_i\}_{i=1}^{m-1}$ to $\{f_{ij}\}_{(i,j)\in I_1}$ and add $g_m$ to
  $\{f_{ij}\}_{(i,j)\in I_2}$.  Then neither sets span the space and so $\{W_i\}_{i=1}^M$ does not do phase retrieval - a
  contradiction.
  \vskip12pt
  \noindent {\bf Case 2:}  $K_1=K_2 < N-1$.
  \vskip12pt
By the argument of Case 1, neither of our sets span (because $K_1 \le N-2$) and so $\{W_i\}_{i=1}^M$ does not do
phase retrieval - a contradiction.
  \end{proof}

As a partial converse to Theorem \ref{lem200}, we have the following proposition for the one-dimensional case.

\begin{proposition}\label{lem201}
  Assume $\{f_i\}_{i=1}^M$ are vectors in $\mathcal{H}_N$ with the following property:

  Whenever we partition $\{f_i\}_{i=1}^M$ into two non-spanning subsets $\{f_i\}_{i\in I_1}$, $\{f_i\}_{i\in I_2}$ then
  each of these sets of vectors spans a hyperplane.  
	
	Then there is an open dense set
  of vectors $f_0\in \mathcal{H}_N$ so that
  $\{f_i\}_{i=0}^M$ does phase retrieval.
\end{proposition}

\begin{proof}
  Choose any vector from the open dense set of
  vectors $f_0$ which is not in $\lspan(\{f_i\}_{i\in J})$ whenever this family does not span the space.
  If we partition $\{f_i\}_{i=0}^N$ into $\{f_i\}_{i\in I_j}$, $j=1,2$, then one of these sets must span the space.  I.e.  If
  neither set spans the space, then by removing the vector $f_0$ each family of vectors spans a hyperplane and one of them must contain $f_0$ and hence spans the space.  \end{proof}

	The following corollary is a rephrasing of Theorem \ref{lem200} in such a way which may be more useful when proving a collection of subspaces fail phase retrieval. 
	
	\begin{corollary}\label{cor987}
	Suppose $\{W_i\}_{i=1}^M$ are subspaces in $\mathcal{H}_N$. Let $\{\varphi_{ij}\}_{j=1}^{I_i}$ be an orthonormal basis for $W_i$ for each $i=1,\dots,M$ and suppose there exists a partition of $\{\varphi_{ij}\}_{j=1,i=1}^{I_i,M}$ into two non-spanning sets $F_1,F_2$. If $\dimm(\lspan  F_1) \leq M-2$ then for all subspaces $W_{M+1}$ of $\mathcal{H}_N$, $\{W_i\}_{i=1}^{M+1}$ fails phase retrieval.
	\end{corollary}
	
	We would like to know if the converse of Corollary \ref{cor987} is true. This is explicitly stated in the following problem.

\begin{problem}
Let $\{W_i\}_{i=1}^M$ be subspaces in $\mathcal{H}_N$. If for all subspaces $W_{M+1}$ of $\mathcal{H}_N$, $\{W_i\}_{i=1}^{M+1}$ fails phase retrieval then does there exist an orthonormal basis $\{\varphi_{ij}\}_{j=1}^{I_i}$ of $W_i$ and a partition $F_1,F_2$ of $\{\varphi_{ij}\}_{j=1,i=1}^{I_i,M}$ such that $\dimm(\lspan  F_1) \leq M-2$ and $\dimm(\lspan  F_2) \leq M-1$?
\end{problem}

In light of Theorem \ref{lem200}, we now further analyze the two hyperplanes spanned by a partition of orthonormal bases for our subspaces which fail phase retrieval and find properties of the vectors within these hyperplanes.

\begin{proposition}
Let $\{W_i\}_{i=1}^{M+1}$ in $\mathcal{H}_N$ yield phase retrieval. Assume $\{\varphi_{ij}\}_{j=1}^{L_i}$ is an orthonormal basis for $W_i$, for $i=1,\dots,M$. Assume there exists a partition $I_1, I_2$ of $\{\varphi_{ij}\}_{j=1,i=1}^{L_i,M}$ so that $\dimm(\lspan (I_s))=N-1, s=1,2$. Then there exists an orthonormal basis $\{\psi_{ij}\}_{j=1}^{L_i}$ for $W_i$, for $i= 1,\dots,M$ and a partition $J_1, J_2$ of $\{\psi_{ij}\}_{j=1,i=1}^{L_i,M}$ satisfying:
\begin{enumerate}
\item $\dimm(\lspan (J_s))=N-1$ for $s=1,2$, and
\item $J_1$ contains at most one vector from $\{\psi_{ij}\}_{j=1}^{L_i}$ for each $i=1,\dots,M$.
\end{enumerate}
\end{proposition}

\begin{proof}
Let $K_s=\lspan(I_s)$ for $s=1,2$. Without loss of generality we can assume that there does not exist any $\varphi_{ij} \in I_1$ in $K_2$ or else we can move these vectors without affecting $K_1, K_2$. (Note that for any vector $\varphi_{ij}$ in $I_1$ which lies in $K_2$, we can move this vector to $I_2$ without affecting $K_2$. Also, this will not decrease dim$(K_1)$ because if it did then we would be left with two non-spanning sets with dimensions $N-2$ and $N-1$ respectively, contradicting the fact that they both must span hyperplanes.). So $K_s$ is a hyperplane for $s=1,2$.

Case 1: If $I_1$ contains at most one vector from $\{\varphi_{ij}\}_{j=1}^{L_i}$ for each $i=\{1,\dots,M\}$ then we are done.

Case 2: Fix $i$. Assume $I_1$ contains at least two vectors from $\{\varphi_{ij}\}_{j=1}^{L_i}$ (so $\dimm(\lspan(W_i)) \geq 2$). Let $W_i'= \lspan\{\varphi_{ij}| \varphi_{ij} \in I_1$ for some $j \in \{1,\dots, L_i\}\}$. Now $\codimm(\lspan_{W_i}(K_2 \cap W_i'))=1$. Choose an orthonormal basis $\{\psi_{ij}\}_{j=1}^{T_i}$ for $K_2 \cap W_i'$ and choose $\psi_{i,T_i+1} \perp (K_2 \cap W_i')$ and $\psi_{i,T_i+1}\in W_i'$. Throw away all $\{\varphi_{ij}\}_{j=1}^{L_i}$ which lie in $W_i'$. Put $\psi_{i,T_i+1}$ into $I_1$ and put the rest of $\{\psi_{ij}\}_{j=1}^{T_i}$ into $I_2$. Repeat this for each $i \in \{1,\dots,M\}$. Define these new sets to be $J_1, J_2$ respectively, and we are done.
\end{proof}

\begin{corollary}\label{lem1.abc}
For each $n\in N$ let $\{\varphi_{i}^{n}\}_{i=1}^{k}$ be a sequence in $\mathcal{H}_{N}$. If 
\[\dimm(\lspan \{\varphi_{i}^{n}\}_{i=1}^{k})=N-1\] for all $n\in N$ and $\varphi_{i}^{n}\to\varphi_{i}$ for each $i\in\{1,\ldots,k\}$, then 
\[\dimm( \lspan \{\varphi_{i}\}_{i=1}^{k})\leq N-1.\]
\end{corollary}

{\bf Acknowledgment:} All four authors were supported by:  NSF DMS 1307685; NSF ATD 1321779; and AFOSR:  FA9550-11-1-0245

\bibliographystyle{amsalpha}

\end{document}